\documentclass[10pt,a4paper]{article}
\usepackage[margin=1.1cm]{geometry}
\usepackage{amsmath,amsthm,amsfonts,amssymb,amscd,cite,graphicx}

\usepackage{titlesec}
\titleformat{\section}
{\normalfont\fontsize{12}{15}\bfseries}{\thesection}{1em.}{}

\newtheorem{proposition}{Proposition}[section]

\newtheorem{problem}{Problem}[section]

\newtheorem{theorem}{Theorem}[section]

\allowdisplaybreaks[4]

\let\oldbibliography\thebibliography
\renewcommand{\thebibliography}[1]{%
  \oldbibliography{#1}%
  \setlength{\itemsep}{-2pt}%
}

\begin{document}

\baselineskip=0.20in

\noindent
{\large \bf Perfect Roman domination in middle graphs}\\

\noindent
Kijung Kim$^{1,}\footnote{Corresponding author (knukkj@pusan.ac.kr)}$ \\

\noindent
\footnotesize $^1${\it Department of Mathematics, Pusan National University, Busan 46241, Republic of Korea}\\
\noindent

\noindent
 (\footnotesize Received: Day Month 202X. Received in revised form: Day Month 202X. Accepted: Day Month 202X. Published online: Day Month 202X.)\\

\setcounter{page}{1} \thispagestyle{empty}

\baselineskip=0.20in

\normalsize

 \begin{abstract}
 \noindent
 The middle graph $M(G)$ of a graph $G$ is the graph obtained by subdividing each edge of $G$ exactly once and joining all these newly introduced vertices of adjacent edges of $G$.
 A perfect Roman dominating function on a graph $G$ is a function $f : V(G) \rightarrow \{0, 1, 2\}$ satisfying the condition that every vertex $v$ with $f(v)=0$ is adjacent to exactly one vertex $u$ for which $f(u)=2$.
 The weight of a perfect Roman dominating function $f$ is the sum of weights of vertices.
 The perfect Roman domination number is the minimum weight of a perfect Roman dominating function on $G$.
 In this paper, we give a characterization of middle graphs with equal Roman domination and perfect Roman domination numbers.
 \\[2mm]
 \end{abstract}

\baselineskip=0.20in

\section{Introduction}

Let $G=(V,E)$ be an undirected graph with the vertex set $V=V(G)$ and edge set $E=E(G)$.
The \textit{order} of $G$ is defined as the cardinality of $V$.
The \textit{open neighborhood} of $v \in V(G)$ is the set $N(v) = \{ u \in V(G) \mid uv \in E(G)\}$.
In \cite{HY}, Hamada and Yoshimura defined the middle graph of a graph.
The \textit{middle graph} $M(G)$ of a graph $G$ is the graph obtained by subdividing each edge of $G$ exactly once
and joining all these newly introduced vertices of adjacent edges of $G$.
The precise definition of $M(G)$ is as follows.
The vertex set $V(M(G))$ is $V(G) \cup E(G)$.
Two vertices $v, w \in V(M(G))$ are adjacent in $M(G)$ if
(i) $v, w \in E(G)$ and $v, w$ are adjacent in $G$ or
(ii) $v \in V(G)$, $w \in E(G)$ and $v, w$ are incident in $G$.

The study of Roman domination was motivated by the defense strategies used to defend the Roman Empire during the reign of Emperor Constantine the Great, 274--337 AD. The concept of Roman domination was introduced in \cite{CDHH,RR,S}.
A function $f : V(G) \rightarrow \{0, 1, 2\}$ is a \textit{Roman dominating function} on $G$ if
every vertex $v \in V(G)$ for which $f(v)=0$ is adjacent to at least one vertex $u \in V(G)$ for which $f(u)=2$.
The weight of a Roman dominating function is the value $\omega(f):= \sum_{v \in V(G)}f(v)$.
The \textit{Roman domination number} $\gamma_R(G)$ is the minimum weight of a Roman dominating function on $G$.
As a variant of Roman domination, a function $f : V(G) \rightarrow \{0, 1, 2\}$ is a \textit{perfect Roman dominating function} on $G$ if
every vertex $v \in V(G)$ for which $f(v)=0$ is adjacent to exactly one vertex $u \in V(G)$ for which $f(u)=2$.
Similarly, perfect Roman domination number $\gamma_{pR}(G)$ is defined.
In \cite{Henning2018}, Henning et al. introduced the notion of perfect Roman domination and showed that if $T$ is a tree on $n \geq 3$ vertices, then $\gamma_{pR}(T) \leq \frac{4}{5}n$.
In \cite{Darkooti}, Darkooti et al. proved that it is NP-complete to decide whether a graph has a perfect Roman dominating function, even if the graph is bipartite.
This suggests determining the exact value of perfect Roman domination numbers for special classes of graphs.
Recently, Kim proved the following result.

\begin{theorem}[\cite{Kim}]\label{main:graphn}
Let $G$ be a graph of order $n$.
Then $\gamma_R(M(G))=n$.
\end{theorem}

Based on this result, we characterize all graphs $G$ such that $\gamma_{R}(M(G)) = \gamma_{pR}(M(G))$.
We try to determine the exact value of perfect Roman domination numbers in middle graphs of special classes of graphs.

\section{Main Results}\label{sec-2}

In this section, we introduce the concept of a middle Roman dominating function to study the Roman domination number of the middle graph $M(G)$ for a given graph $G$.
A \textit{middle Roman dominating function} (MRDF) on a graph $G$ is a function $f: V \cup E \rightarrow \{0, 1, 2\}$ satisfying
the following conditions:
(i) every element $x \in V$ for which $f(x)=0$ is incident to at least one element $y \in E$ for which $f(y)=2$,
(ii) every element $x \in E$ for which $f(x)=0$ is adjacent or incident to at least one element $y \in V \cup E$ for which $f(y)=2$.
A MRDF $f$ gives an ordered partition $(V_0 \cup E_0, V_1 \cup E_1, V_2 \cup E_2)$ (or $(V_0^f \cup E_0^f, V_1^f \cup E_1^f, V_2^f \cup E_2^f)$ to refer to $f$) of $V \cup E$, where $V_i := \{ x \in V \mid f(x)=i \}$ and $E_i := \{ x \in E \mid f(x)=i \}$.
The weight of a middle Roman dominating function $f$ is $\sum_{x \in V \cup E} f(x)$.
The \textit{middle Roman domination number} $\gamma_R^\star(G)$ of $G$ is the minimum weight of a middle Roman dominating function of $G$.
A \textit{$\gamma_R^\star(G)$-function} is a MRDF on $G$ with weight $\gamma_R^\star(G)$.
Similarly, we can define a perfect middle Roman dominating function (PMRDF) and related definitions.
For a subset $S$ of $G$, the subgraph obtained from $G$ by deleting all vertices in $S$ and all edges incident with $S$ is denoted by $G - S$.
For terminology and notation on graph theory not given here, the reader is referred to \cite{Bondy08}.
We make use of the following result.

\begin{proposition}[\cite{Yue-2020}]\label{pro-0}
Let $G$ be a graph with components $G_1, \dotsc, G_t$.
Then $\gamma_{pR}(G) = \sum_{i=1}^t \gamma_{pR}(G_i)$.
\end{proposition}

The following is our main theorem.

\begin{theorem}\label{thm-1}
Let $G$ be a graph.
Then $\gamma_R^\star(G) = \gamma_{pR}^\star(G)$ if and only if there exists a $\gamma_R^\star(G)$-function such that
(i) vertices incident to edges in $E_2$ are adjacent to vertices in $V_1$ and
(ii) $G - \{u, v \in V(G) \mid uv \in E_2 \}$ is an empty graph.
\end{theorem}

\begin{proof}
If $G$ is an empty graph, then the statement holds.
By Proposition \ref{pro-0}, from now on we assume that $G$ is connected and not empty.

($\Rightarrow$):
Let $f=(V_0 \cup E_0, V_1 \cup E_1, V_2 \cup E_2)$ be a $\gamma_R^\star(G)$-function  such that $\gamma_R^\star(G) = \gamma_{pR}^\star(G)$.
We proceed by proving six claims.

\vskip5pt
\textbf{Claim 1.} $V_2 = \emptyset$.

Suppose that there exists $v \in V_2$.
Consider $G - \{v\}$.
Define $g : V(G - \{v\}) \cup E(G - \{v\}) \rightarrow \{0, 1, 2\}$ by $g(u) = \text{min}\{f(u) + f(uv), 2\}$ for each $u \in N(v)$
and $g(x) =f(x)$ otherwise.
Then $g$ is a MRDF on $G - \{v\}$ with $\omega(g)=n-2$.
Since $G - \{v\}$ has order $n-1$, by Theorem \ref{main:graphn} $\omega(g) \geq \gamma_R^\star(G - \{v\})=n-1$, a contradiction.

\vskip5pt
\textbf{Claim 2.} $E_1 = \emptyset$.

Suppose that there exists $e \in E_1$.
Let $u$ and $v$ be vertices incident to $e$.
We divides the following three cases depending on the values of $u$ and $v$ assigned under $f$.

\vskip5pt
Case 1. $u, v \in V_0$.
There exist $e_1, e_2 \in E_2$ such that $u, v$ are incident to $e_1, e_2$, respectively.
Also, there exist $u', v' \in V(G)$ such that $u', v'$ are incident to $e_1, e_2$, respectively.
Consider $G - \{u, u', v, v'\}$.
Define $g : V(G - \{u, u', v, v'\}) \cup E(G - \{u, u', v, v'\}) \rightarrow \{0, 1, 2\}$ by $g(x) =f(x)$.
Since there is no edge in $E_2$ adjacent to $e_1$ or $e_2$, $g$ is a MRDF on $G - \{u, u', v, v'\}$.
Then $g$ is a MRDF with $\omega(g)\leq n-5$.
Since $G - \{u, u', v, v'\}$ has order $n-4$, by Theorem \ref{main:graphn} $\omega(g) \geq \gamma_R^\star(G - \{u, u', v, v'\})=n-4$, a contradiction.

\vskip5pt
Case 2. $u \in V_0, v \in V_1$ or $v \in V_0, u \in V_1$.
By symmetry, assume that $u \in V_0$ and $v \in V_1$.
There exists $e_1 \in E_2$ incident to $u$.
Also, there exist $u' \in V(G)$ incident to $e_1$.
Consider $G - \{u, u', v\}$.
Define $g : V(G - \{u, u', v\}) \cup E(G - \{u, u', v\}) \rightarrow \{0, 1, 2\}$ by $g(x) = \text{min}\{f(x) + f(xv), 2\}$ for each $x \in N(v) \setminus \{u\}$ and $g(x) =f(x)$ otherwise.
Then $g$ is a MRDF with $\omega(g) \leq n-4$.
Since $G - \{u, u', v\}$ has order $n-3$, by Theorem \ref{main:graphn} $\omega(g) \geq \gamma_R^\star(G - \{u, u', v\})=n-3$, a contradiction.

\vskip5pt
Case 3. $u, v \in V_1$.
Consider $G - \{u, v\}$.
Define $g : V(G - \{u, v\}) \cup E(G - \{u, v\}) \rightarrow \{0, 1, 2\}$ by $g(x) = \text{min}\{f(x) + f(xu), 2\}$ for each $x \in N(u) \setminus \{v\}$, $g(x) = \text{min}\{f(x) + f(xv), 2\}$ for each $x \in N(v) \setminus \{u\}$ and $g(x) =f(x)$ otherwise.
Then $g$ is a MRDF with $\omega(g) \leq n-3$.
Since $G - \{u, v\}$ has order $n-2$, by Theorem \ref{main:graphn} $\omega(g) \geq \gamma_R^\star(G - \{u, v\})=n-2$, a contradiction.

\vskip5pt
\textbf{Claim 3.} Every edge in $E_2$ is incident to vertices in $V_0$.

Let $e \in E_2$ and $e=uv$.
Suppose that $u \not\in V_0$ or $v \not\in V_0$.
Consider $G - \{u, v\}$.
Define $g : V(G - \{u, v\}) \cup E(G - \{u, v\}) \rightarrow \{0, 1, 2\}$ by $g(x) = \text{min}\{f(x) + f(xu), 2\}$ for each $x \in N(u) \setminus \{v\}$, $g(x) = \text{min}\{f(x) + f(xv), 2\}$ for each $x \in N(v) \setminus \{u\}$ and $g(x) =f(x)$ otherwise.
Then $g$ is a MRDF with $\omega(g)\leq n-3$.
Since $G - \{u, v\}$ has order $n-2$, by Theorem \ref{main:graphn} $\omega(g) \geq \gamma_R^\star(G - \{u, v\})=n-2$, a contradiction.

\vskip5pt
\textbf{Claim 4.} Every edge in $E_2$ is adjacent to edges in $E_0$.

By Claims 2, 3 and the hypothesis that $f$ is a PMRDF, Claim 4 follows.

\vskip5pt
\textbf{Claim 5.} Vertices incident to edges in $E_2$ are adjacent to vertices in $V_1$.

Let $uv=e \in E_2$.
Suppose that $u$ is adjacent to $w \in V_0$.
Then $w$ must be incident to some $e' \in E_2$.
Since $uw \in E_0$ is adjacent to $e'$, this is a contradiction.

\vskip5pt
\textbf{Claim 6.} $G - \{u, v \in V(G) \mid uv \in E_2 \}$ is an empty graph.

If $H:= G - \{u, v \in V(G) \mid uv \in E_2 \}$ is not empty, then every edge of $H$ must be assigned $1$ under $f$.
The weight of $f$ is not equal to the order $G$, a contradiction.

\vskip10pt
($\Leftarrow$):
The conditions (i) and (ii) imply that the $\gamma_R^\star(G)$-function is a PMRDF on $G$.
\end{proof}

Based on Theorem \ref{thm-1},
we determine the exact values of perfect Roman domination numbers for middle graphs of paths and cycles.

\begin{proposition}\label{prop-1}
For a path $P_n$ of order $n$, $\gamma_{pR}^\star(P_n)=n$
\end{proposition}

\begin{proof}
Let $P_n = v_1v_2 \dotsc v_n$.
Clearly $\gamma_{pR}^\star(P_2)=2$.
For $n \geq 3$, we divides the following three cases.

\vskip5pt
Case 1.  $n \equiv 0$ (mod 3).
Define $f : V(P_n) \cup E(P_n) \rightarrow \{0, 1, 2\}$ by
$f(v_{3i+1})=1$, $f(v_{3i+2}v_{3i+3})=2$ for $0 \leq i \leq \frac{n-3}{3}$ and $f(x) = 0$ otherwise.

\vskip5pt
Case 2.  $n \equiv 1$ (mod 3).
Define $f : V(P_n) \cup E(P_n) \rightarrow \{0, 1, 2\}$ by
$f(v_{3i+1})=1$, $f(v_{3i+2}v_{3i+3})=2$ for $0 \leq i \leq \frac{n-4}{3}$, $f(v_n)=1$ and $f(x) = 0$ otherwise.

\vskip5pt
Case 3.  $n \equiv 2$ (mod 3).
Define $f : V(P_n) \cup E(P_n) \rightarrow \{0, 1, 2\}$ by
$f(v_{3i+1}v_{3i+2})=2$, $f(v_{3i+3})=1$ for $0 \leq i \leq \frac{n-5}{3}$, $f(v_{n-1}v_n)=2$ and $f(x) = 0$ otherwise.

In any case, it is easy to see that $f$ is a PMRDF with the weight $n$.
This completes the proof.
\end{proof}

\begin{proposition}\label{prop-2}
For a cycle $C_n$ of order $n$, $\gamma_{pR}^\star(C_n)=n$ if $n \equiv 0$ (mod 3), $n+1$ otherwise.
\end{proposition}

\begin{proof}
Let $C_n = v_1v_2 \dotsc v_nv_1$, and let $f$ be a $\gamma_{pR}^\star(C_n)$-function.
Suppose that $\omega(f) < |V(C_n) \cup E(C_n)|$.
Then there exists an element $x \in V(C_n) \cup E(C_n)$ such that $f(x)=2$.
Without loss of generality, we assume that $f(v_{n-1}v_n)=2$.
Consider $C_n - \{v_{n-1}, v_n\} \cong P_{n-2}$.
If $\gamma_{R}^\star(C_n) = \gamma_{pR}^\star(C_n)$, then
it follows from Theorem \ref{thm-1} that we have $f(v_1)=f(v_{n-2})=1$.

If $n-2 \equiv 1$ (mod 3), then by the Case 2 of Proposition \ref{prop-1}
$P_{n-2}$ has a PMRDF $g$ such that $g(v_1)=g(v_{n-2})=1$ and $\gamma_{pR}^\star(P_{n-2}) =n-2$.
This implies that $\gamma_{pR}^\star(C_n)=n$.

If $n-2 \not\equiv 1$ (mod 3), then Theorem \ref{thm-1} implies that $\gamma_{pR}^\star(C_n) > \gamma_{R}^\star(C_n)$.
Now we can define a PMRDF $f$ with the weight $n+1$ by giving $f(v_1v_n)=1$ in the Cases 2 and 3 of Proposition \ref{prop-1}.
Thus, $\gamma_{pR}^\star(C_n) = n+1$.
\end{proof}

Finally, we conclude our paper by suggesting the following problems.
\begin{problem}
For a complete bipartite graph $K_{m,n}$, what is the exact value of $\gamma_{pR}^\star(K_{m,n}) ?$
\end{problem}

\begin{problem}
For a complete graph $K_n$, what is the exact value of $\gamma_{pR}^\star(K_n) ?$
\end{problem}

\section*{Acknowledgment}

This research was supported by Basic Science Research Program through the National Research Foundation of Korea funded by the Ministry of Education (2020R1I1A1A01055403).


\end{document}